\documentclass[11pt]{amsart}

\usepackage{geometry}
\usepackage{palatino}
\usepackage[all]{xy}
\usepackage{amsthm,amssymb,amsmath}
\usepackage{tikz}
\usepackage{mathrsfs,amscd,longtable}
\usepackage{multirow}
\newtheorem{theorem}{Theorem}[section]
\newtheorem{conjecture}[theorem]{Conjecture}

\newtheorem{definition'}[theorem]{`Definition'}
\newtheorem{lemma}[theorem]{Lemma}
\newtheorem{proposition}[theorem]{Proposition}
\newtheorem{example}[theorem]{Example}

\newcommand{\bb}[1]{\mathbb{#1}}
\newcommand{\mc}[1]{\mathcal{#1}}
\newcommand{\mf}[1]{\mathfrak{#1}}
\begin{document}

\title[Lusztig-Vogan Bijection]{Some Calculations of the Lusztig-Vogan Bijection for Classical Nilpotent Orbits}
\author{Kayue Daniel Wong}

\begin{abstract}
In this manuscript, we compute explicitly the Lusztig-Vogan bijection for local systems of some classical, special, nilpotent orbits. Using these results, we prove a conjecture of Achar and Sommers on regular functions of some covers of classical nilpotent orbits.
\end{abstract}
\maketitle
\section{Introduction}
\subsection{Unipotent Representations and Quantization}
Let $G$ be a complex simple Lie group. In \cite{BV1985}, Barbasch and Vogan studied the \textbf{special unipotent representations} of $G$, which are of utmost interest in various aspects of representation theory. For instance, they are related to Arthur's packet of automorphic forms, and are conjectured to be unitary. More specifically, they are also conjectured to be `building blocks' of the unitary dual of $G$. Indeed, in \cite{B1989}, Barbasch generalized the idea of special unipotent representations to \textbf{unipotent representations}, which are used to classify the unitary dual of classical Lie groups.

Another interesting application of special unipotent representations is their relations with the Orbit Method, first introduced by A.A. Kirillov. Roughly speaking, for any (co)adjoint orbit $\mathcal{O}$ of a Lie algebra $\mf{g}$, one would like to `attach' a (preferably unitary) representation to $\mathcal{O}$. This idea is pursued nicely when $\mf{g}$ is a nilpotent or solvable Lie algebra, but several difficulties came up when $\mf{g}$ is semisimple. In the context of nilpotent coadjoint orbits in a semisimple Lie algebra $\mf{g}$, the idea of Orbit Method suggests the following:
\begin{conjecture} \label{conj:orbitmethod}
Let $\mc{O}$ be a nilpotent orbit in $\mf{g}$, and $R(\mc{O})$ be the ring of regular function of $\mc{O}$, then there is a (not necessarily unique) $(\mf{g}_{\bb{C}},K_{\bb{C}})$-module $Q$ such that
$$Q|_{K_{\bb{C}}} \cong R(\mc{O}).$$
(note that $K \leq G$ is the maximal compact subgroup of $G$, hence its complexification $K_{\bb{C}}$ is isomorphic to $G$). More generally, let $e \in \mc{O}$ and $G_e$ be the isotropy group of $e$ with connected component $(G_e)^0$. Then for any irreducible representation $\rho$ of the component group $A(\mc{O}) := G_e/(G_e)^0$, there exists a $(\mf{g}_{\bb{C}},K_{\bb{C}})$-module $Q_{\rho}$ such that
$$Q_{\rho}|_{K_{\bb{C}}} \cong R(\mc{O}, \rho) = Ind_{G^e}^{G}(\rho),$$
where $R(\mc{O}, \rho)$ is the global section of the $G$-equivariant bundle $G {\times}_{G_e} V_{\rho} \to G/G^e \cong \mc{O}$. In particular, when $\rho = \mathrm{triv}$ is the trivial representation, then $R(\mc{O}, \mathrm{triv}) = R(\mc{O})$.
\end{conjecture}

To relate the above conjecture with unipotent representations, recall in \cite{BV1985} that all special unipotent representations of $\mf{g}$ are parametrized by the set
$$\mc{N}_{o,\widehat{\overline{a}}} := \{(\mc{O}, \pi) \ |\ \mc{O} \text{ is a special nilpotent orbit, } \pi \text{ is an irreducible representation of } \overline{A}(\mc{O}) \},$$
where $\overline{A}(\mc{O})$ is the \textbf{Lusztig quotient} of the component group $A(\mc{O})$ \cite[Section 4]{BV1985}. For each $(\mc{O},\pi) \in \mc{N}_{o,\widehat{\overline{a}}}$, we write $X_{\mc{O},\pi}$ be its corresponding special unipotent representation. Then we can state a conjecture of Vogan in the context of complex semisimple Lie groups:
\begin{conjecture}[\cite{V3}, Conjecture 12.1] \label{conj:unip}
Suppose $\mc{O}$ is a special nilpotent orbit. For every irreducible representation $\pi$ of $\overline{A}(\mc{O})$, there exists an irreducible representation $\rho$ of $A(\mc{O})$ such that
$$X_{\mc{O}, \pi}|_{K_{\bb{C}}} \cong R(\mc{O}, \rho) = Ind_{G^e}^{G}(\rho).$$
\end{conjecture}
For classical nilpotent orbits, Barbasch showed the following:
\begin{theorem}[\cite{B2}, Theorem 4.10.1] \label{thm:unipclassical}
Let $G$ be a complex simple Lie group of classical type. Then Conjecture \ref{conj:unip} holds for all special orbits satisfying $A(\mc{O}) = \overline{A}(\mc{O})$.
\end{theorem}
More precisely, Barbasch showed there is a one-to-one correspondence between the $R(\mc{O},\rho)$'s and $X_{\mc{O},\pi}$'s for special classical orbits satisfying $A(\mc{O}) = \overline{A}(\mc{O})$. We will see in Section \ref{subsec:generators} that the correspondence is trivial, i.e. $\rho = \pi$ as representations of $A(\mc{O}) = \overline{A}(\mc{O})$. Moreover, it is not hard to extend Theorem \ref{thm:unipclassical} to all classical special orbits without the condition on component groups. In other words:\\

\noindent{\bf Theorem A} (Section \ref{subsec:generalize}).
{\it Conjecture \ref{conj:unip} holds for all special classical orbits.}\\

For exceptional nilpotent orbits, the results in \cite{McG2} verified Conjecture \ref{conj:unip} for $G_2$. However, it is left unproved for other exceptional groups. In an upcoming work, the author will study the conjecture for exceptional Richardson orbits.

\subsection{The Lusztig-Vogan Bijection}
We now focus on another conjecture of Vogan, which is related to $R(\mc{O}, \sigma) \cong Ind_{G^e}^{G}(\sigma)$ for all possible irreducible representations $\sigma \in \widehat{G^e}$. Write $\Lambda^+(G) \subset \mathfrak{t}^*$ as the collection of highest dominant weights of finite dimensional representations of $G$. By Theorem 8.2 of \cite{V1998}, $R(\mc{O},\sigma)$ can be expressed uniquely in the form
\begin{align} \label{eq:lusztigvogan}
R(\mc{O},\sigma) = Ind_{G^e}^G(\sigma) = \sum_{\lambda \in \Lambda^+(G)} m_{\lambda}(\mathcal{O},\sigma) Ind_{T}^{G}(e^{\lambda}),
\end{align}
where all but finitely many $m_{\lambda}(\mathcal{O},\sigma) \in \mathbb{Z}$ are zero.
\begin{definition'}\label{def:gamma}
Let $$\mc{N}_{o,\widehat{e}} = \{ (\mc{O},\sigma)| \mc{O}\ \text{is a nilpotent orbit,}\ \sigma \in \widehat{G_e} \},$$
then the {\bf Lusztig-Vogan map} is defined to be
$$\Gamma: \mc{N}_{o,\widehat{e}} \to \Lambda^+(G),$$
where $\Gamma(\mathcal{O},\sigma) = \lambda_{max}$ is the maximal element among all $m_{\lambda}(\mathcal{O},\sigma) \neq 0$ in Equation (\ref{eq:lusztigvogan}).
\end{definition'}
Note that by the $W$-symmetry of weights for finite dimensional $G$-modules, any expressions of the form $\sum_{\sigma \in \Lambda(G)} a_{\lambda} Ind_{T}^{G}(e^{\lambda})$ can be $W$-conjugated such that each summand lies in $\Lambda^+(G)$. Therefore, $\Gamma(\mathcal{O},\sigma)$ can be represented by any of its $W$-conjugates.

A priori, $\Gamma$ is not well-defined. Under a different definition from $\Gamma$ above, Bezrukavnikov \cite{Be} proved that there is a bijection between the two sets. The work of Achar \cite{Ac1}, \cite{Ac2} showed that the two definitions are the same for Type $A$ orbits, and explicitly computed the bijection for all Type $A$ orbits. In this manuscript, we only study $\Gamma$ as defined above. In fact, given that Conjecture \ref{conj:unip} holds, one can study $\Gamma$ using unipotent representations $X_{\mc{O},\pi}$. Also, all unipotent representations $X_{\mc{O},\pi}$ can also be expressed in the form of Equation (\ref{eq:lusztigvogan}) (see Example \ref{eg:unip}). One may therefore wish to `define' the following:
\begin{definition'}\label{def:mu}
Let $\Psi$ be the map
$$\Psi: \mc{N}_{o,\widehat{\overline{a}}} \to \Lambda^+(G),$$
where $\Psi(\mc{O},\pi)$ is the maximal element of $X_{\mc{O},\pi}$ expressed in the form of Equation (\ref{eq:lusztigvogan}).
\end{definition'}
As in Definition \ref{def:gamma}, one does not know whether $\Psi$ is well-defined. Indeed, we have:\\

\noindent {\bf Theorem B} (Theorem \ref{prop:psigeneral}).
{\it Let $G = SO(2n+1,\bb{C})$, $Sp(2n,\bb{C})$ or $SO(2n,\bb{C})$. Then $\Psi$ is well-defined and injective, and can be computed explicitly for all special orbits $\mc{O}$ and $\pi \in \widehat{\overline{A}(\mc{O})}$.}\\

Note that by Theorem A, $\Psi(\mc{O},\pi) = \Gamma(\mc{O}, \rho)$ for all classical, special nilpotent orbits $\mc{O}$. Therefore, $\Gamma$ can be computed explicitly in these cases.

\subsection{A Conjecture of Achar and Sommers} \label{subsec:AS}
We now relate Theorem B to a conjecture of Achar and Sommers in \cite{AS}. Let $\mathcal{N}_o$ be the set of all nilpotent orbits in a classical Lie algebra $\mathfrak{g}$, ${}^L{\mathcal{N}_o}$ be the set of all nilpotent orbits in the Langlands dual ${}^L\mathfrak{g}$. In \cite{S2}, Sommers constructed a surjective map
$$d: \mathcal{N}_{o,c}\twoheadrightarrow {}^L\mathcal{N}_o,$$
where $\mathcal{N}_{o,c} = \{(\mathcal{O},C) | \mathcal{O} \in \mathcal{N}_o, C \subset \overline{A}(\mathcal{O})\ \text{conjugacy class} \}$. In fact, Sommers explicitly described the surjection by assigning a {\bf canonical preimage} to each nilpotent orbit $\mc{O}^{\vee} \in {}^L\mathcal{N}_o$.

For classical groups, $\overline{A}(\mathcal{O}) \cong (\mathbb{Z}/2\mathbb{Z})^q$ is abelian, so every conjugacy class $C$ is a single element. For a choice of generators $\{\theta_q, \dots, \theta_1\}$ of $\overline{A}(\mc{O})$, let $C = \Pi_{i \in I} \theta_i$ for some subset $I \subset \{q,q-1,\dots,1\}$. Define
$$K_C := \langle \theta_i | i \in I \rangle \leq \overline{A}(\mathcal{O}),$$
and consider the preimage of $K_C$ under the quotient map $r: A(\mathcal{O}) \to \overline{A}(\mathcal{O})$, i.e. $H_C := r^{-1}(K_C)$. Then $H_C$, as a subgroup of the $G$-equivariant component group $A(\mathcal{O})$, corresponds to an orbit cover $\widetilde{\mathcal{O}}^{C} \cong G/G_C$ of $\mathcal{O}$. The conjecture of Achar and Sommers is given by:

\begin{conjecture}(\cite{AS}, Conjecture 3.1) \label{conj:AS}
Suppose $\mc{O}^{\vee}$ is a classical nilpotent orbit in ${}^{L}\mathcal{N}_o$ with canonical preimage $(\mc{O},C)$. Writing
\begin{align} \label{eq:rocchar}
R(\widetilde{\mathcal{O}}^C) \cong Ind_{G_C}^{G}(\mathrm{triv}) = \sum_{\lambda \in \Lambda^+} m_{\lambda} Ind_{T}^G(\lambda)
\end{align}
as in the form of Equation \eqref{eq:lusztigvogan}, then the maximal element in the expression is equal to $h^{\vee}$, the semisimple element of a Jacobson-Morozov triple of $\mc{O}^{\vee}$.
\end{conjecture}

Our last main Theorem gives an affirmative answer of the conjecture:\\

\noindent {\bf Theorem C} (Section \ref{sec:AS}).
{\it Conjecture \ref{conj:AS} holds for all orbits in ${}^L\mathcal{N}_o$.}


\section{Nilpotent Orbits in Classical Lie Algebras}
\subsection{Basic Setup}
We begin by studying special orbits in the classical Lie algebras. Recall in \cite{CM} that all classical nilpotent orbits $\mc{O}$ can be described by Young diagrams satisfying certain properties which we will describe below. We use square bracket $[r_l \geq r_{l-1} \geq \dots \geq r_1]$ to denote a Young diagram of size $n$ in terms of rows, and round bracket $(c_l \geq c_{l-1} \geq \dots \geq c_1)$ in terms of columns.

\begin{proposition} \label{prop:classrows}
The classification of classical special nilpotent orbits in terms of rows is given as follows:
\begin{itemize}
\item {\bf Type $B_n$:} Let $G = SO(2n+1,\bb{C})$, then all nilpotent orbits $\mc{O}$ in $\mf{g} = Lie(G)$ are parametrized by Young diagrams $[r_{2k} \geq r_{2k-1} \geq \dots \geq r_0]$ of size $2n+1$ such that every even number appears an even number of times among the $r_i$'s. The orbit $\mc{O}$ is special iff its {\bf transpose} $(r_{2k} \geq r_{2k-1} \geq \dots \geq r_0)$ defines a nilpotent orbit of Type $B_n$. In other words, the even rows of $\mc{O}$ must occur in the form $r_{2l-1} = r_{2l-2} = 2b$.

\item {\bf Type $C_n$:} Let $G = Sp(2n,\bb{C})$, then all nilpotent orbits $\mc{O}$ in $\mf{g} = Lie(G)$ are parametrized by Young diagrams $[r_{2k+1} \geq r_{2k} \geq \dots \geq r_1]$ of size $2n$ such that every odd number appears an even number of times among the $r_i$'s (take $r_1 = 0$ if necessary). The orbit $\mc{O}$ is special iff its transpose $(r_{2k+1} \geq r_{2k} \geq \dots \geq r_1)$ defines a nilpotent orbit of Type $C_n$. In other words, the odd rows of $\mc{O}$ must occur in the form $r_{2l-1} = r_{2l-2} = 2c+1$.

\item {\bf Type $D_n$:} Let $G = SO(2n,\bb{C})$, then all nilpotent orbits $\mc{O}$ in $\mf{g} = Lie(G)$ are parametrized by Young diagrams $[r_{2k+1} \geq r_{2k} \geq \dots \geq r_0]$ of size $2n$ such that every even number appears an even number of times among the $r_i$'s. The only exceptions are the following -- if the Young diagram is of the form
$$[2\alpha_k, 2\alpha_k, 2\alpha_{k-1}, 2\alpha_{k-1},\dots, 2\alpha_1, 2\alpha_1],$$
i.e. the diagram is {\bf very even}, then there are two orbits $\mc{O}_I$, $\mc{O}_{II}$ attached to this diagram. These orbits are called {\bf very even orbits}. The orbit $\mc{O}$ is special iff its transpose $(r_{2k+1} \geq r_{2k} \geq \dots \geq r_0)$ defines a nilpotent orbit of Type $C_n$. In other words, the even rows of $\mc{O}$ must occur in the form $r_{2l-1} = r_{2l-2} = 2d$. In particular, all very even orbits are special.
\end{itemize}
\end{proposition}

\begin{proof}
The description of classical nilpotent orbits and special nilpotent orbits are given in Section 5 and 6.3 of \cite{CM} respectively. And the last statement for each type follows from Proposition \ref{prop:class} below.
\end{proof}

From now on, we write $\mc{O} = [r_l \geq r_{l-1} \geq \dots \geq r_1]$ to denote a non-very even nilpotent orbit $\mc{O}$ whose partition is given by $[r_l \geq r_{l-1} \geq \dots r_1]$, and write
\[
\mc{O}_{I,II} = [2\alpha_k, 2\alpha_k, 2\alpha_{k-1}, 2\alpha_{k-1},\dots, 2\alpha_1, 2\alpha_1]_{I,II}\]
for the two very even orbits corresponding to a very even Young diagram. Also, we use $\mc{O}^{\bf t} = (r_l \geq r_{l-1} \geq \dots \geq r_1)$ to denote the transpose of $\mc{O}$.

\begin{proposition} \label{prop:abarrows}
The component group $A(\mc{O})$ and the Lusztig quotient $\overline{A}(\mc{O})$ of the classical special nilpotent orbits are given as follows:
\begin{itemize}
\item {\bf Type $B_n$:} Let $G = SO(2n+1,\bb{C})$ and $\mc{O} = [r_{2k} \geq r_{2k-1} \geq \dots \geq r_0]$ be a special orbit. Then $A(\mc{O}) \cong (\bb{Z}/2\bb{Z})^p$, where $p$ is one less than the number of distinct odd $r_i$'s showing up in $\mc{O}$. For the Lusztig quotient, we separate all even rows (which, by our classification of special orbits in Proposition \ref{prop:classrows}, must be of the form $r_{2l-1} = r_{2l-2} = \alpha$), along with odd row pairs of the form $r_{2l} = r_{2l-1} = \beta$ and get
\begin{align*}
    \mc{O} = &[r_{2q}'' > r_{2q-1}'' \geq r_{2q-2}'' > \dots \geq r_2'' > r_1'' \geq r_0'']\\
		&\cup [\alpha_1, \alpha_1, \dots, \alpha_x, \alpha_x] \cup [\beta_1, \beta_1, \dots, \beta_y, \beta_y],
\end{align*}
    then $\overline{A}(\mc{O}) \cong (\bb{Z}/2\bb{Z})^q$.\\ 
\item {\bf Type $C_n$:}  Let $G = Sp(2n,\bb{C})$ and $\mc{O} = [r_{2k+1} \geq r_{2k} \geq \dots \geq r_1]$ be a special orbit. Then $A(\mc{O}) \cong (\bb{Z}/2\bb{Z})^p$, where $p$ is the number of distinct even $r_i$'s showing up in $\mc{O}$. For the Lusztig quotient, we separate all odd rows (which, by our classification of special orbits in Proposition \ref{prop:classrows}, must be of the form $r_{2l-1} = r_{2l-2} = \alpha$), and even row pairs of the form $r_{2l} = r_{2l-1} = \beta$ and get
\begin{align*}
    \mc{O} = &[r_{2q+1}'' \geq r_{2q}'' > r_{2q-1}'' \geq \dots > r_3'' \geq r_2'' > r_1'']\\
		&\cup [\alpha_1, \alpha_1, \dots, \alpha_x, \alpha_x] \cup [\beta_1, \beta_1, \dots, \beta_y, \beta_y],
\end{align*}
    then $\overline{A}(\mc{O}) \cong (\bb{Z}/2\bb{Z})^q$.\\ 
\item {\bf Type $D_n$:}  Let $G = SO(2n,\bb{C})$ and $\mc{O} = [r_{2k+1} \geq r_{2k} \geq \dots \geq r_0]$ be a special, non-very even orbit. Then $A(\mc{O}) \cong (\bb{Z}/2\bb{Z})^p$, where $p$ is one less than the number of distinct odd $r_i$'s showing up in $\mc{O}$. For the Lusztig quotient, we separate all even rows (which, by our classification of special orbits in Proposition \ref{prop:classrows}, must be of the form $r_{2l-1} = r_{2l-2} = \alpha$), and all odd row pairs $r_{2l} = r_{2l-1} = \beta$ and get
  \begin{align*}
	\mc{O} = &[r_{2q+1}'' \geq r_{2q}'' > r_{2q-1}'' \geq \dots \geq r_2'' > r_1'' \geq r_0'']\\
	&\cup [\alpha_1, \alpha_1, \dots, \alpha_x, \alpha_x] \cup [\beta_1, \beta_1, \dots, \beta_y, \beta_y],
	\end{align*}
    then $\overline{A}(\mc{O}) \cong (\bb{Z}/2\bb{Z})^q$. Moreover, if $\mc{O}_I$ and $\mc{O}_{II}$ are very even, then all $A(\mc{O}_I)$, $A(\mc{O}_{II})$, $\overline{A}(\mc{O}_I)$, $\overline{A}(\mc{O}_{II})$ are trivial.
\end{itemize}
\end{proposition}

\begin{proof}
The description of $A(\mc{O})$ for each classical type is stated in Section 6.1 of \cite{CM}, and the generators of $\overline{A}(\mc{O})$ is stated in Section 5 of \cite{S2}. More precisely, in Type $B_n$ and $D_n$, if $r_{2i}'' > r_{2i-1}'' = r_{2i-2}''$, then $r_{2i-2}'' \in S_{even}$ (in the notations of \cite{S1}, \cite{S2} - see proof of Proposition \ref{prop:abar} for some examples) contributes a generator in $\overline{A}(\mc{O})$. If $r_{2i} > r_{2i-1}'' > r_{2i-2}''$, then the two distinct numbers $r_{2i-1}''$,  $r_{2i-2}''$ belong to $S_{odd}$, and $r_{2i-2}''$ can also be chosen as a generator of $\overline{A}(\mc{O})$. 

The arguments are similar in Type $C_n$. If $r_{2i+1}'' = r_{2i}'' > r_{2i-1}''$, then $r_{2i}'' \in S_{even}$ and it contributes a generator of $\overline{A}(\mc{O})$. If $r_{2i+1}'' > r_{2i}'' > r_{2i-1}''$, then $r_{2i+1}''$, $r_{2i}'' \in S_{odd}$ and $r_{2i}''$ can also be chosen as a generator of $\overline{A}(\mc{O})$. 
\end{proof}


It is sometimes easier to state our results using columns of the Young diagrams attached to $\mc{O}$. We give the classification of special nilpotent orbits, along with their component group $A(\mc{O})$ and the Lusztig quotient $\overline{A}(\mc{O})$ as follows:
\begin{proposition} \label{prop:class}
The classification of classical special nilpotent orbits in terms of columns is given as follows:
\begin{itemize}
\item {\bf Type $B_n$:} Let $G = SO(2n+1,\bb{C})$, then all nilpotent orbits in $\mf{g} = Lie(G)$ are parametrized by Young diagrams $\mc{O} = (a_{2k+1} \geq a_{2k} \geq \dots \geq a_0)$ of size $2n+1$ such that $a_{2l}+a_{2l-1}$ is even for all $l$ (We insist that there are {\bf even number of columns}, by taking $a_0 = 0$ if necessary).\\
The orbit $\mc{O}$ is special if all $a_i$'s are odd, or the columns of even sizes occur only in the form $a_{2l} = a_{2l-1} = 2b$ (note that this forces $a_0 = 0$).\\

\item {\bf Type $C_n$:} Let $G = Sp(2n,\bb{C})$, then all nilpotent orbits in $\mf{g} = Lie(G)$ are parametrized by Young diagrams $\mc{O} = (a_{2k} \geq a_{2k-1} \geq \dots \geq a_0)$ of size $2n$ such that $a_{2l}+a_{2l-1}$ is even for all $l$ (We insist that there are {\bf odd number of columns}, by taking $a_0 = 0$ if necessary). \\
The orbit $\mc{O}$ is special if all $a_i$'s are even, or the columns of odd sizes occur only in the form $a_{2l} = a_{2l-1} = 2c+1$.\\

\item {\bf Type $D_n$:} Let $G = SO(2n,\bb{C})$, then all nilpotent orbits in $\mf{g} = Lie(G)$ are parametrized by Young diagrams $\mc{O} = (a_{2k+1} \geq a_{2k} \geq \dots \geq a_0)$ of size $2n$ such that $a_{2l}+a_{2l-1}$ is even for all $l$ (We insist that there are {\bf even number of columns}, by taking $a_0 = 0$ if necessary). The only exceptions are the following - if the Young diagram is of the form
$$(2\alpha_k, 2\alpha_k, 2\alpha_{k-1}, 2\alpha_{k-1},\dots, 2\alpha_1, 2\alpha_1),$$
i.e. the diagram is {\bf very even}, then there are two orbits $\mc{O}_I$, $\mc{O}_{II}$ attached to this diagram. These orbits are called {\bf very even orbits}.\\
The orbit $\mc{O}$ is special if all $a_i$'s are even, or the columns of odd sizes occur only in the form $a_{2l} = a_{2l-1} = 2d+1$. In particular, all very even orbits are special.
\end{itemize}
\end{proposition}

\begin{proposition} \label{prop:abar}
The component group $A(\mc{O})$ and the Lusztig quotient $\overline{A}(\mc{O})$ of the classical special nilpotent orbits in terms of columns are given as follows:
\begin{itemize}
\item {\bf Type $B_n$:} Let $G = SO(2n+1,\bb{C})$ and $\mc{O} = (a_{2k+1} \geq a_{2k} \geq \dots \geq a_0)$ be a special orbit. Separate all column pairs $a_{2m+1} = a_{2m} = \nu$ and get
\begin{align*}
    \mc{O} = (a_{2p+1}' \geq a_{2p}' \geq \dots \geq a_0') \cup (\nu_1, \nu_1, \dots, \nu_y, \nu_y),
\end{align*}
then $A(\mc{O}) \cong (\bb{Z}/2\bb{Z})^p$. For the Lusztig quotient, we further separate all even column pairs $a_{2l}' = a_{2l-1}' = \mu$ and get
\begin{align*}
    \mc{O} = (a_{2q+1}'' \geq a_{2q}'' \geq \dots \geq a_0'')\cup (\mu_1, \mu_1, \dots, \mu_x, \mu_x) \cup (\nu_1, \nu_1, \dots, \nu_y, \nu_y),
\end{align*}
then $\overline{A}(\mc{O}) \cong (\bb{Z}/2\bb{Z})^q$.\\ 
		
\item {\bf Type $C_n$:} Let $G = Sp(2n,\bb{C})$ and $\mc{O} = (a_{2k} \geq a_{2k-1} \geq \dots \geq a_0)$ be a special orbit. Separate all column pairs $a_{2m+1} = a_{2m} = \nu$ and get
    $$\mc{O} = (a_{2p}' \geq a_{2p-1}' \geq \dots \geq a_0') \cup (\nu_1, \nu_1, \dots, \nu_y, \nu_y),$$
    then $A(\mc{O}) \cong (\bb{Z}/2\bb{Z})^p$. For the Lusztig quotient, we further separate all odd column pairs $a_{2l}' = a_{2l-1}' = \mu$ and get
    $$\mc{O} = (a_{2q}'' \geq a_{2q-1}'' \geq \dots \geq a_0'')\cup (\mu_1, \mu_1, \dots, \mu_x, \mu_x) \cup (\nu_1, \nu_1, \dots, \nu_y, \nu_y),$$
    then $\overline{A}(\mc{O}) \cong (\bb{Z}/2\bb{Z})^q$.\\ 

\item {\bf Type $D_n$:} Let $G = SO(2n,\bb{C})$ and $\mc{O} = (a_{2k+1} \geq a_{2k} \geq \dots \geq a_0)$ be a special, non-very even orbit. Separate all column pairs $a_{2m+1} = a_{2m} = \nu$ and get
    $$\mc{O} = (a_{2p+1}' \geq a_{2p}' \geq \dots \geq a_0') \cup (\nu_1, \nu_1, \dots, \nu_y, \nu_y),$$
    then $A(\mc{O}) \cong (\bb{Z}/2\bb{Z})^p$. For the Lusztig quotient, we further separate all odd column pairs $a_{2l}' = a_{2l-1}' = \mu$ and get
    $$\mc{O} = (a_{2q+1}'' \geq a_{2q}'' \geq \dots \geq a_0'')\cup (\mu_1, \mu_1, \dots, \mu_x, \mu_x) \cup (\nu_1, \nu_1, \dots, \nu_y, \nu_y),$$
    then $\overline{A}(\mc{O}) \cong (\bb{Z}/2\bb{Z})^q$. Moreover, if $\mc{O}_I$ and $\mc{O}_{II}$ are very even, then all $A(\mc{O}_I)$, $A(\mc{O}_{II})$, $\overline{A}(\mc{O}_I)$, $\overline{A}(\mc{O}_{II})$ are trivial.

\end{itemize}
\end{proposition}

\begin{example} \mbox{}\\
(a) Let $\mc{O} = (9,7,5,5,3,2,2,2,2,0)$ be a special nilpotent orbit of Type $B_n$. Then $\mc{O} = (9,7,3,0)$ $\cup$ $(2,2)$ $\cup (5,5,2,2)$ and $A(\mc{O}) \cong (\bb{Z}/2\bb{Z})^2$, $\overline{A}(\mc{O}) \cong \bb{Z}/2\bb{Z}$.

\noindent (b) Let $\mc{O} = (6,4,4,2,2,2,2)$ be a special nilpotent orbit of Type $C_n$. Then $\mc{O} = (6)$ $\cup$ $\phi$ $\cup$ $(4,4,2,2,2,2)$ and $A(\mc{O}) =   \overline{A}(\mc{O}) = \{e\}$.

\noindent (c) Let $\mc{O} = (6,3,3,2,2,2,2,0)$ be a special nilpotent orbit of Type $D_n$. Then $\mc{O} = (6,2,2,0)$ $\cup$ $(3,3)$ $\cup$ $(2,2)$ and $A(\mc{O}) \cong (\bb{Z}/2\bb{Z})^2$,   $\overline{A}(\mc{O}) \cong \bb{Z}/2\bb{Z}$.
\end{example}

\begin{proof}
The proof for Type $C_n$ is given in Section 2 of \cite{W3}. The proofs for Type $B_n$ and $D_n$ can be done similarly. If fact, using the results of Type $C_n$, we can prove the proposition for Type $D_n$: Let
\begin{align*}
\mc{O} &= (a_{2k+1} \geq a_{2k} \geq \dots \geq a_{2l+2} \geq a_{2l+1} = a_{2q+1}'' \geq a_{2l} \geq \dots \geq a_0)\\
&= (a_{2q+1}'' \geq a_{2q}'' \geq \dots \geq a_0'')\cup (\mu_1, \mu_1, \dots, \mu_x, \mu_x) \cup (\nu_1, \nu_1, \dots, \nu_y, \nu_y)
\end{align*}
be a non-very even orbit of Type $D_n$. By the constructions in the Proposition, $(a_{2k+1} \geq a_{2k} \geq$ $\dots \geq a_{2l+2})$ must consist of column pairs $(\nu, \nu)$ of even sizes.

Recall the parametrization of $\overline{A}(\mc{O})$ in Section 5 of \cite{S2}. In fact, the size of $\overline{A}(\mc{O})$ depends only on the ordering of the odd rows of $\mc{O}$. Therefore, the columns $(a_{2k+1} \geq a_{2k} \geq$ $\dots \geq a_{2l+2})$ in $\mc{O}$ do not contribute to $\overline{A}(\mc{O})$ as in our Proposition. So we can reduce our study of $\overline{A}(\mc{O})$ to the orbit
$$\mc{P} = (a_{2q+1}'' = a_{2l+1} \geq a_{2l} \geq \dots \geq a_0).$$
Removing the longest column $a_{2q+1}'' = a_{2l+1}$ from $\mc{P}$, then $\mc{Q} = (a_{2l} \geq \dots \geq a_0)$ defines a special nilpotent orbit of Type $C_n$. Using the notations in \cite{S2}, we divide the even row sizes of $\mc{Q}$ into those appearing odd number of times and even number of times respectively:
\begin{align*}
S_{\mc{Q}, odd} &= \{ 2\alpha_{2m} > 2\alpha_{2m-1} > \dots > 2\alpha_1 \}\ \text{,putting }\alpha_1 = 0\text{ if necessary};\\
S_{\mc{Q}, even} &= \{ 2\beta_l > 2\beta_{l-1} > \dots > 2\beta_1 \}.
\end{align*}
Then $\overline{A}(\mc{Q})$ is generated by all $2\alpha_{2r-1}$ such that $\alpha_{2r-1} \neq 0$, and all $2\beta_s$ satisfying $2\alpha_{2r+1} > 2\beta_s > 2\alpha_{2r}$.

We now look at the odd row sizes of $\mc{P}$. If $\alpha_1 = 0$, then
\begin{align*}
S_{\mc{P}, odd} &= \{ 2\alpha_{2m}+1 > 2\alpha_{2m-1}+1 > \dots > 2\alpha_1+1 = 1 \};\\
S_{\mc{P}, even} &= \{ 2\beta_l+1 > 2\beta_{l-1}+1 > \dots > 2\beta_1+1 \}.
\end{align*}
On the other hand, if $\alpha_1 \neq 0$, then
\begin{align*}
S_{\mc{P}, odd} &= \{ 2\alpha_{2m}+1 > 2\alpha_{2m-1}+1 > \dots > 2\alpha_1+1 \};\\
S_{\mc{P}, even} &= \{ 2\beta_l > 2\beta_{l-1} > \dots > 2\beta_1 > 1 \}.
\end{align*}
Following the parametrization of $\overline{A}(\mc{P})$ for Type $D$ orbits in \cite{S2}, one can see that $\overline{A}(\mc{Q}) = \overline{A}(\mc{P})$ in both cases. This matches with the statement of the Proposition.
\end{proof}

\subsection{Generators of the Lusztig quotient} \label{subsec:generators}
Using the notations of Section 3 in \cite{S3}, we describe a choice of generators of $\overline{A}(\mc{O}) \cong (\bb{Z}/2\bb{Z})^q$ explicitly:

For Type $B_n$ and Type $D_n$ orbits, elements in $A(\mc{O})$ are parametrized by an even number of commuting order 2 elements $b_k$'s, with $k$ equal to $a_{2q+1}''$, $a_{2q-1}''$, $\dots$, $a_{1}''$ or $\mu_1$, $\dots$, $\mu_x$. By Proposition 4 of \cite{S2}, if $a_{2i+1}'' \geq \mu_w \geq a_{2i-1}''$, then $b_{a_{2i+1}''}b_{\mu_w}$ descends to the trivial element in $\overline{A}(\mc{O})$. So the elements $\theta_{i} := b_{a_{2i+1}''}b_{a_{2i-1}''}$, $i = q, \dots, 1$ are generators of $\overline{A}(\mc{O})$.

For Type $C_n$ orbits,  elements in $A(\mc{O})$ are parametrized by a collection of commuting order 2 elements $b_k$'s, with $k$ equal to $a_{2q-1}''$, $a_{2q-3}''$, $\dots$, $a_1''$ or $\mu_1$, $\dots$, $\mu_x$. If $\mu_w \geq a_{2q}''$, then $b_{\mu_w}$ descends to the trivial element in $\overline{A}(\mc{O})$. Also, if
$a_{2i+1}'' \geq \mu_w \geq a_{2i-1}''$, then $b_{a_{2i+1}''}b_{\mu_w}$ descends to the trivial element in $\overline{A}(\mc{O})$. So the elements $\theta_q := b_{a_{2q-1}''}$, $\theta_{i} := b_{a_{2i+1}''}b_{a_{2i-1}''}$, $i = q-1, \dots, 1$ are generators of $\overline{A}(\mc{O})$. Our choice of generators is compatible to that in \cite{S3}.

Using our generators of $\overline{A}(\mc{O})$, one can define generators of irreducible representations $\varphi_i$ of $\overline{A}(\mc{O})$ by setting $\varphi_i(\theta_j)$ $:= (-1)^{\delta_{ij}}$ for all $i$ and $j$. On the other hand, if $A(\mc{O}) = \overline{A}(\mc{O})$, then Section 5.3 of \cite{B2} used the notation $((a_{2q-1}'')_{\epsilon_q}(a_{2q-3}'')_{\epsilon_{q-1}}\dots (a_{1}'')_{\epsilon_1})$, with $\epsilon_i \in \{ +,- \}$ to describe all irreducible representations of $\overline{A}(\mc{O})$ and all special unipotent representations attached to $\mc{O}$. We identify
\begin{align} \label{eq:identify}
\varphi := \prod_{i \in I} \varphi_i \longleftrightarrow ((a_{2q-1}'')_{\gamma_q} (a_{2q-3}'')_{\gamma_{q-1}} \dots (a_{1}'')_{\gamma_1})
\end{align}
where $I$ is a subset of $\{q,q-1,\dots,1\}$, with $\gamma_i = -$ if $i \in I$ and $\gamma_j = +$ if $j \notin I$. By Theorem 0.4 of \cite{Lu} and Propositions 4.14-4.16 of \cite{BV1985}, the above identification is natural in the sense that they parametrize the same unipotent representation, i.e.
\[
X_{\mc{O},\varphi} = \pi((a_{2q-1}'')_{\gamma_q}(a_{2q-3}'')_{\gamma_{q-1}}\dots (a_{1}'')_{\gamma_1})
\]
for all $\varphi \in \widehat{\overline{A}(\mc{O})}$. Moreover, the main result in Section 5.3 of \cite{B2} implies that $\pi((a_{2q-1}'')_{\gamma_q}$ $(a_{2q-3}'')_{\gamma_{q-1}}$ $\dots$ $(a_{1}'')_{\gamma_1})$ is equal to $R(\mc{O},\varphi)$. Therefore, we have $X_{\mc{O},\pi} \cong R(\mc{O},\pi)$ for all special orbits satisfying $A(\mc{O}) \cong \overline{A}(\mc{O})$.

\begin{example} \label{eg:unip}
Let $\mc{O} = (4,4,2,2,0) \cup \phi \cup \phi$ be a nilpotent orbit of Type $C_{6}$. We follow the recipe in \cite[Section 3]{W3} to find the special unipotent representations attached to $\mc{O}$. Firstly, its Spaltenstein-Lusztig dual is given by $\mc{O}^{\vee} = [5,3,3,1,1]$ and the infinitesimal character of $X_{\mc{O},\pi}$ is given by
$$\frac{1}{2}h^{\vee} = (2,1,1,0; 1,0).$$
On the other hand, the {\bf special piece} attached to $\mc{O}$ (\cite{Lu}) contains the orbits
$$\mc{O}_1 = (4,4,2,2,0);\ \ \mc{O}_2 = (5,3,2,2,0);\ \ \mc{O}_3 = (4,4,3,1,0);\ \ \mc{O}_4 = (5,3,3,1,0).$$
By Theorem 0.4 and Section 1 of \cite{Lu}, these orbits correspond to the elements
$$[e, 1];\ \ [\theta_2, 1];\ \ [\theta_1, 1];\ \ [\theta_2\theta_1, 1]$$
in $M(\mc{O}) := \overline{A}(\mc{O}) \times \widehat{\overline{A}(\mc{O})}$ respectively under the notations of \cite[Definition 4.6]{BV1985}.

For all $\mc{O}_i$, their corresponding Springer representations are given by $\sigma_i = j_{W_i}^{W(C_{6})}(\mathrm{sgn})$ for $1 \leq i \leq 4$, where
$$W_1 = C_2 \times D_2 \times C_1 \times D_1;\ \ W_2 = D_3 \times C_1 \times C_1 \times D_1;$$
$$W_3 = C_2 \times D_2 \times D_2 \times C_0;\ \ W_4 = D_3 \times C_1 \times D_2 \times C_0.$$
Therefore, the special unipotent representations are given by
$$X_{\mc{O},\mathrm{triv}} = \frac{1}{4}(R_e + R_{\theta_{2}} + R_{\theta_{1}} + R_{\theta_{2}\theta_{1}});\ \ X_{\mc{O},\varphi_2} = \frac{1}{4}(R_e - R_{\theta_{2}} + R_{\theta_{1}} - R_{\theta_{2}\theta_{1}});$$
$$X_{\mc{O},\varphi_1} = \frac{1}{4}(R_e + R_{\theta_{2}} - R_{\theta_{1}} - R_{\theta_{2}\theta_{1}});\ \ X_{\mc{O},\varphi_2\varphi_1} = \frac{1}{4}(R_e - R_{\theta_{2}} - R_{\theta_{1}} + R_{\theta_{2}\theta_{1}}),$$
with
\begin{align*}
R_e &= \sum_{w \in W(C_2 \times D_2 \times C_1 \times D_1)} \mathrm{sgn}(w)Ind_T^G((21;10;1;0)-w(21;10;1;0));\\
R_{\theta_{2}} &= \sum_{w \in W(D_3 \times C_1 \times C_1 \times D_1)} \mathrm{sgn}(w)Ind_T^G((210;1;1;0)-w(210;1;1;0));\\
R_{\theta_{1}} &= \sum_{w \in C_2 \times D_2 \times D_2 \times C_0)} \mathrm{sgn}(w)Ind_T^G((21;10;10)-w(21;10;10));\\
R_{\theta_{2}\theta_{1}} &= \sum_{w \in W(D_3 \times C_1 \times D_2 \times C_0)} \mathrm{sgn}(w)Ind_T^G((210;1;10)-w(210;1;10)).
\end{align*}
It is easy to see that
$$X_{\mc{O},\mathrm{triv}} = \pi(4_+2_+),\ \ X_{\mc{O},\varphi_2} = \pi(4_-2_+),\ \ X_{\mc{O},\varphi_1} = \pi(4_+2_-),\ \ X_{\mc{O},\varphi_2\varphi_1} = \pi(4_-2_-)$$
as observed in the paragraph prior to this example.

Moreover, Equations (39)-(40) of \cite{B2} suggests that $X_{\mc{O},\pi} = R(\mc{O},\pi)$ for all irreducible representations $\pi$ of $\overline{A}(\mc{O})$: For example, Equation (39) of {\it loc. cit.} says that $\pi(4_-2_-) = R(\mc{O},\rho)$, where $\rho(b_4) = -1$ and $\rho(b_2) = 1$ (note that the notations for $R(\mc{O},\rho)$ in {\it loc. cit.} are in terms of rows, whereas our notations $b_k$ are in terms of columns). That is, $\rho(\theta_2) = \rho(b_4) = -1$, $\rho(\theta_1) = \rho(b_4b_2) = -1$, i.e. $\rho = \varphi_2\varphi_1$ and hence $X_{\mc{O},\varphi_2\varphi_1} = \pi(4_-2_-) = R(\mc{O},\varphi_2\varphi_1)$. We note that the same results can also be obtained using Section 4.2 of \cite{S3}.
\end{example}

\section{Lusztig-Vogan Bijection} \label{sec:A}
\subsection{Vogan's Conjecture for classical nilpotent orbits} \label{subsec:generalize}
In this subsection, we study how Theorem \ref{thm:unipclassical} can be extended to orbits with $A(\mc{O}) \neq \overline{A}(\mc{O})$. Using the descriptions of special nilpotent orbits in Proposition \ref{prop:abar}, these are the orbits $\mc{O}$ {\it with} the $\mu_i$ columns. In order to prove an analogous result of Theorem \ref{thm:unipclassical} for these orbits, it suffices to consider $\mc{O}$ {\it without} the $\nu_j$ columns as in Section 2 of \cite{B2}. In other words, we study orbits of the form
$$\mc{O} = \mc{O}'' \cup (\mu_1, \mu_1, \dots, \mu_x, \mu_x) \cup \phi,$$
where $\mc{O}'' = \mc{O}'' \cup \phi \cup \phi$ is a special nilpotent orbit without the $\mu$ or $\nu$ entries. In this case, $\mc{O}$ is induced from $\mc{O}''$ by
$$\mc{O} = Ind_{\mf{g}'' \oplus \mf{gl}(\mu_1) \oplus \dots \oplus \mf{gl}(\mu_x)}^{\mf{g}} (\mc{O}'' \oplus 0 \oplus \dots \oplus 0),$$
and $A(\mc{O}) \cong \overline{A}(\mc{O}) \times (\bb{Z}/2\bb{Z})^x$, so every irreducible representation $\rho$ of $A(\mc{O})$ can be represented as $\rho = \pi \boxtimes \chi_1' \boxtimes \dots \boxtimes \chi_x'$, where $\pi$ is an irreducible representation of $\overline{A}(\mc{O}) \cong \overline{A}(\mc{O}'') \cong A(\mc{O}'')$ and $\chi_i'$ is either $\mathrm{triv}$ or $\mathrm{sgn}$.

\begin{lemma}
Let $\mc{O}^+ = \mc{O}^{\Delta} \cup (\mu_1, \mu_1, \dots, \mu_x, \mu_x) \cup \phi$, where $\mc{O}^{\Delta}$ is a triangular orbit given in Section 9 of \cite{BV1985}. Then $X_{\mc{O}^+,\pi} \cong R(\mc{O}^+, \pi \boxtimes \mathrm{triv}^x)$, where $\mathrm{triv}^x$ is the shorthand for $\overbrace{\mathrm{triv} \boxtimes \dots \boxtimes \mathrm{triv}}^{x\ terms}$.
\end{lemma}
\begin{proof}
Note that $\mc{O}^+$ is an even orbit. By Section 8 of \cite{BV1985},
$$X_{\mc{O}^+,\pi} = Ind_{G' \times GL(\mu_1) \times \dots \times GL(\mu_x)}^G(X_{\mc{O}^{\Delta},\pi} \boxtimes \mathrm{triv}^x).$$
By the structure of $X_{\mc{O}^{\Delta},\pi}$ and induction in stages, $X_{\mc{O}^+,\pi}$ is isomorphic to:
\begin{align*}
&\text{Type }B_n:\ Ind_{SO(2k+1) \times GL(2k-1) \times \dots \times GL(1) \times GL(\mu_1) \times \dots \times GL(\mu_x)}^G(\mathrm{triv} \boxtimes \pi_{2k-1} \boxtimes \dots \boxtimes \pi_{1} \boxtimes \mathrm{triv}^{x});\\
&\text{Type }C_n:\ Ind_{GL(2k) \times \dots \times GL(2) \times GL(\mu_1) \times \dots \times GL(\mu_x)}^G(\pi_{2k} \boxtimes \dots \boxtimes \pi_{2} \boxtimes \mathrm{triv}^{x});\\
&\text{Type }D_n:\ Ind_{SO(2k+2) \times GL(2k) \times \dots \times GL(2) \times GL(\mu_1) \times \dots \times GL(\mu_x)}^G(\mathrm{triv} \boxtimes \pi_{2k} \boxtimes \dots \boxtimes \pi_{2} \boxtimes \mathrm{triv}^{x}),
\end{align*}
where every $\pi_i$ is either trivial or determinant representation depending on $\pi$. At the same time, Section 4.2 of \cite{S3} says the right hand side is equal to $R(\mc{O}^+, \pi \boxtimes \mathrm{triv}^x)$. So the result follows.
\end{proof}

With the above Lemma, one can generalize the results of Theorem \ref{thm:unipclassical} using the same arguments as in \cite{B2}: \\

\noindent {\it Proof of Theorem A}. Suppose $\mc{O} = \mc{O}'' \cup (\mu_1, \dots, \mu_x) \cup \phi$ is a special orbit, which we can induce $\mc{O}$ suitably to $\mc{O}^+ = \mc{O}^{\Delta} \cup (\mu_1, \dots, \mu_x) \cup \phi$. By \cite{V3}, $X_{\mc{O},\pi} = R(\mc{O},\rho_{\pi}) - Y_{\pi}$ for some genuine module $Y_{\pi}$ whose support is strictly smaller than $\overline{\mc{O}}$. Using Proposition 4.5.1 of \cite{B2}, we have
\begin{align} \label{eq:ind}
Ind(X_{\mc{O},\pi}) = Ind(R(\mc{O},\rho_{\pi})) - Ind(Y_{\pi}) = \sum_{\rho' \in Ind_{A(\mc{O})}^{A(\mc{O}^+)}(\rho_{\pi})} R(\mc{O}^+, \rho') - Z_{\pi} - Ind(Y_{\pi}),
\end{align}
On the other hand, one can check directly from the character formula that $Ind(X_{\mc{O},\pi}) = \sum_{\pi^+ \in Ind^{\overline{A}(\mc{O}^+)}_{\overline{A}(\mc{O})}(\pi)} X_{\mc{O}^+,\pi^+}$. According to the previous lemma, each summand of $Ind(X_{\mc{O},\pi})$ is isomorphic to $R(\mc{O}^+, \pi^+ \boxtimes \mathrm{triv}^x)$. By summing up Equation \eqref{eq:ind} for all irreducible representations $\pi$, and the linear independence of $\{ R(\mc{Q},\sigma)\ |\ \mc{Q} = G\cdot e\ \text{nilpotent orbit},\ \sigma \in \widehat{G^e}\}$ given in \cite{V1998}, we argue as in \cite{B2} that $Ind(Y_{\pi}) = 0$ and $Z_{\pi} = 0$ for all $\pi$.

To see that $\rho_{\pi} = \pi \boxtimes \mathrm{triv}^x$, we go back to Equation \eqref{eq:ind} again with $Z_{\pi} = Ind(Y_{\pi}) = 0$, i.e.
$$\sum_{\pi^+ \in Ind^{\overline{A}(\mc{O}^+)}_{\overline{A}(\mc{O})}(\pi)} R(\mc{O}^+, \pi^+ \boxtimes \mathrm{triv}^x) = \sum_{\rho' \in Ind_{A(\mc{O})}^{A(\mc{O}^+)}(\rho_{\pi})} R(\mc{O}^+, \rho').$$
By linear independence of the $R(\mc{O},\sigma)$'s, one must have $Ind^{\overline{A}(\mc{O}^+)}_{\overline{A}(\mc{O})}(\pi) \boxtimes \mathrm{triv}^x$ $= Ind^{A(\mc{O}^+)}_{A(\mc{O})}(\pi \boxtimes \mathrm{triv}^x)$ $= Ind_{A(\mc{O})}^{A(\mc{O}^+)}(\rho_{\pi})$. This forces $\rho_{\pi} = \pi \boxtimes \mathrm{triv}^x$ as required. \qed

Consequently, the map $\Psi$ `defined' in the Introduction is related to the Lusztig-Vogan map by
$$\Psi(\mc{O},\pi) = \Gamma(\mc{O}, \pi \boxtimes \mathrm{triv}^x),$$
where $\overline{A}(\mc{O})$ is realized as a quotient of $A(\mc{O})$ by omitting the last $x$ coordinates of $A(\mc{O})$. In the next two subsections, we will compute $\Psi$ explicitly for all special orbits and all irreducible representations $\pi$ of $\overline{A}(\mc{O})$.

\subsection{Proof of Theorem B for orbits without $\mu$'s and $\nu$'s} \label{subsec:nomunu}
We first write down $\Psi(\mc{O}'',\pi)$ for $\mc{O}''$ that does not contain any $\mu$'s or $\nu$'s by the description of special orbits in Proposition \ref{prop:abar}. Since $\overline{A}(\mc{O}'')$ is equal to $(\bb{Z}/2\bb{Z})^q$, we can write
$$\pi = \chi_q \boxtimes \chi_{q-1} \boxtimes \dots \boxtimes \chi_1,$$
where each $\chi_i$ is either $\mathrm{triv}$ or $\mathrm{sgn}$ on $\theta_{i} \in \overline{A}(\mc{O}'')$. Let $S$ be a subset of $\{1,2,\dots,q\}$ such that $\chi_s = \mathrm{sgn}$ for all $s \in S$ and $\chi_t = \mathrm{triv}$ for all $t \notin S$. Then we have the description of $\Psi(\mc{O}'',\pi)$ in the following proposition:

\begin{proposition} \label{prop:psinomunu} \mbox{}\\
\begin{itemize}
\item {\bf Type $B_n$:} Let $G = SO(2n+1,\bb{C})$ and
$$\mc{O}'' = (a_{2q+1}'' \geq a_{2q}'' \geq \dots \geq a_0'')\cup \phi \cup \phi$$
be a special nilpotent orbit of Type $B_n$ (such that $a_0'' = 0$) with $A(\mc{O}'') = \overline{A}(\mc{O}'')$. Then
$$\Psi(\mc{O}'',\pi)= (a_{2q+1}''-2,a_{2q+1}''-4,\dots,1) \cup \bigcup_{s \in S} A_s' \cup \bigcup_{t \notin S} B_t',$$
where\\
$A_s' = \begin{cases}
    (a_{2s}'',a_{2s}''-2,\dots,a_{2s-1}'',a_{2s-1}''-2,a_{2s-1}''-2,\dots,3,3,1,1) \text{ if }\ \frac{a_{2s}'' - a_{2s-1}''}{4} \in \mathbb{N}\\
    (a_{2s}'',a_{2s}''-2,\dots,a_{2s-1}''+2,a_{2s-1}''-1,a_{2s-1}''-1,\dots,2,2,0)\ \ \ \text{ otherwise }.
\end{cases}$\\
$B_t' = \begin{cases}
    (a_{2s}'',a_{2s}''-2,\dots,a_{2s-1}''+2,a_{2s-1}''-1,a_{2s-1}''-1,\dots,2,2,0) \text{ if }\ \frac{a_{2t}'' - a_{2t-1}''}{4} \in \mathbb{N}\\
    (a_{2s}'',a_{2s}''-2,\dots,a_{2s-1}'',a_{2s-1}''-2,a_{2s-1}''-2,\dots,3,3,1,1)\ \ \  \text{ otherwise }.
\end{cases}$

For example, if $\mc{O}'' = (11,9,5,3,1,0)$, then $A(\mc{O}'') = (\bb{Z}/2\bb{Z})^2$, generated by $\theta_{2} = b_{11}b_5$ and $\theta_{1} = b_5b_1$. Then
\begin{align*}
\Psi(\mc{O}'',\mathrm{triv} \boxtimes \mathrm{triv}) &= (9,7,5,3,1;\ \ 9,7,5,3,3,1,1;\ \ 3,0);\\
\Psi(\mc{O}'',\mathrm{triv} \boxtimes \mathrm{sgn}) &= (9,7,5,3,1;\ \ 9,7,5,3,3,1,1;\ \ 3,1);\\
\Psi(\mc{O}'',\mathrm{sgn} \boxtimes \mathrm{triv}) &= (9,7,5,3,1;\ \ 9,7,4,4,2,2,0;\ \ 3,0);\\
\Psi(\mc{O}'',\mathrm{sgn} \boxtimes \mathrm{sgn}) &= (9,7,5,3,1;\ \ 9,7,4,4,2,2,0;\ \ 3,1).
\end{align*}

\item {\bf Type $C_n$:} Let $G = Sp(2n,\bb{C})$ and
$$\mc{O}'' = (a_{2q}'' \geq a_{2q-1}'' \geq \dots \geq a_0'')\cup \phi \cup \phi$$ be a special nilpotent orbit of Type $C_n$ with $\overline{A}(\mc{O}'') = (\bb{Z}/2\bb{Z})^q$. Then
    $$\Psi(\mc{O}'',\pi)= \bigcup_{s \in S} A_s \cup \bigcup_{t \notin S} B_t \cup (a_0'',a_0''-2,\dots,2),$$
    where\\
$A_s = \begin{cases}
    (a_{2s}'',a_{2s}''-2,\dots,a_{2s-1}'',a_{2s-1}''-2,a_{2s-1}''-2,\dots,2,2,0) \text{ if }\ \frac{a_{2s}'' - a_{2s-1}''}{4} \in \mathbb{N}\\
    (a_{2s}'',a_{2s}''-2,\dots,a_{2s-1}''+2,a_{2s-1}''-1,a_{2s-1}''-1,\dots,3,3,1,1)\ \ \ \text{ otherwise }.
\end{cases}$\\
$B_t = \begin{cases}
    (a_{2t}'',a_{2t}''-2,\dots,a_{2t-1}''+2,a_{2t-1}''-1,a_{2t-1}''-1,\dots,3,3,1,1) \text{ if }\ \frac{a_{2t}'' - a_{2t-1}''}{4} \in \mathbb{N}\\
    (a_{2t}'',a_{2t}''-2,\dots,a_{2t-1}'',a_{2t-1}''-2,a_{2t-1}''-2,\dots,2,2,0)\ \ \  \text{ otherwise }.
\end{cases}$

For example, if $\mc{O}'' = (10,6,4,2,0)$, then $A(\mc{O}'') = (\bb{Z}/2\bb{Z})^2$, generated by $\theta_{2} = b_6$ and $\theta_{1} = b_2b_6$. Then
\begin{align*}
\Psi(\mc{O}'',\mathrm{triv} \boxtimes \mathrm{triv}) &= (10,8,5,5,3,3,1,1;\ \ 4,2,0);\\
\Psi(\mc{O}'',\mathrm{triv} \boxtimes \mathrm{sgn}) &= (10,8,5,5,3,3,1,1;\ \ 4,1,1);\\
\Psi(\mc{O}'',\mathrm{sgn} \boxtimes \mathrm{triv}) &= (10,8,6,4,4,2,2,0;\ \ 4,2,0);\\
\Psi(\mc{O}'',\mathrm{sgn} \boxtimes \mathrm{sgn}) &= (10,8,6,4,4,2,2,0;\ \ 4,1,1).
\end{align*}

\item {\bf Type $D_n$:} Let
$$\mc{O}'' = (a_{2q+1}'' \geq a_{2q}'' \geq \dots \geq a_0'') \cup \phi \cup \phi$$
be a special, non-very even nilpotent orbit of Type $D_n$ with $\overline{A}(\mc{O}) = (\bb{Z}/2\bb{Z})^q$. Then
$$\Psi(\mc{O}'',\pi)= (a_{2q+1}-2,\dots,2,0) \cup \bigcup_{s \in S} A_s \cup \bigcup_{t \notin S} B_t \cup (a_0'',a_0''-2,\dots,2),$$
where $A_s$, $B_t$ are the same as Type C above.\\
Suppose $\mc{O}_{I,II} = (2\alpha_k, 2\alpha_k, 2\alpha_{k-1}, 2\alpha_{k-1},\dots, 2\alpha_1, 2\alpha_1)_{I,II}$ are the very even orbits. Then $\overline{A}(\mc{O}_{I,II}) = 1$ and
\begin{align*}
 \Psi(\mc{O}_I,\mathrm{triv}) = &\bigcup_{1 \leq l \leq k} (2\alpha_l -1, 2\alpha_l -1,\dots, 3,3,1,1); \\
	\Psi(\mc{O}_{II},\mathrm{triv}) = &\bigcup_{1 \leq l \leq k-1} (2\alpha_l -1, 2\alpha_l -1,\dots, 3,3,1,1)\\
&\cup (2\alpha_k -1, 2\alpha_k -1,\dots, 3,3,1,-1).
\end{align*}
\end{itemize}
\end{proposition}

\begin{proof}
The character formulas of $X_{\mc{O}'',\pi}$ for all special nilpotent orbits $\mc{O}$ and all $\pi \in \widehat{\overline{A}(\mc{O})}$ of Type $C_n$ is given in Section 3 of \cite{W3}. The results of the above Proposition is given in Remark 3.3 of {\it loc. cit.} One can use the same technique in {\it loc. cit.} to obtain the character formulas for special nilpotent orbits in Type $B_n$ and $D_n$ (see Example \ref{eg:unip} below), along with the results of the above Proposition.
\end{proof}

\begin{example} \label{eg:unip2}
We go back to the character formulas obtained in Example \ref{eg:unip}, with $\mc{O} = (4,4,2,2,0)$ of Type $C_{6}$. To compute $\Psi(\mc{O},\mathrm{sgn} \boxtimes \mathrm{triv})$, one needs to find the maximal term appearing in
$$X_{\mc{O},\mathrm{sgn} \boxtimes \mathrm{triv}} = \frac{1}{4}(R_e - R_{\theta_{2}} + R_{\theta_{1}} - R_{\theta_{2}\theta_{1}}) = \frac{1}{2}[\frac{1}{2}(R_e - R_{\theta_{2}}) + \frac{1}{2}(R_{\theta_{1}} - R_{\theta_{2}\theta_{1}})].$$
By the results in Section 4 of \cite{W2}, the first 4 coordinates of both $\frac{1}{2}(R_e - R_{\theta_{2}})$ and $\frac{1}{2}(R_{\theta_{1}} - R_{\theta_{2}\theta_{1}})$ that gives the maximal norm are $(4,2,2,0)$. For the last 2 coordinates, one needs to find the maximal length element of
$$\frac{1}{2}[\sum_{w \in W(C_1 \times D_1)} \mathrm{sgn}(w)Ind_T^G((1;0)-w(1;0)) + \sum_{w \in W(D_2)} \mathrm{sgn}(w)Ind_T^G((10)-w(10))].$$
By applying results in Section 4 of {\it loc. cit.} once more, this is equal to $(1,1)$. So
$$\Psi(\mc{O},\mathrm{sgn} \boxtimes \mathrm{triv}) = (4,2,2,0; 1,1).$$

The calculations of non-very even Type $D_n$ orbits are similar, but the recipe for Type $B_n$ orbits is slightly more complicated, since the infinitesimal characters in this setting have half-integral coordinates. More details can be found in Section 5.1.2 of the Ph.D. thesis of the author \cite{W1}.
\end{example}

\subsection{Proof of Theorem B for all special nilpotent orbits} \label{subsec:gen}
We are now in the position to prove Theorem B for all special nilpotent orbits $\mc{O}$ of classical Type. Note that
$$\mc{O} = \mc{O}'' \cup (\mu_1, \mu_1, \dots, \mu_x, \mu_x) \cup (\nu_1, \nu_1, \dots, \nu_y, \nu_y),$$
where $\mc{O}''$ is as in Section \ref{subsec:nomunu}, and $\mc{O}$ is an induced orbit from $\mc{O}''$ of the form
$$\mc{O} = Ind_{\mf{g}'' \oplus \mf{gl}(\mu_1) \oplus \dots \oplus \mf{gl}(\mu_x) \oplus \mf{gl}(\nu_1) \oplus \dots \oplus \mf{gl}(\nu_y)}^{\mf{g}}(\mc{O}'' \oplus 0 \oplus \dots \oplus 0),$$
where $G$ and $G''$ are classical Lie groups of the same type with $\mc{O}'' \subset \mf{g}''$. By Proposition \ref{prop:abar}, $\overline{A}(\mc{O}) = \overline{A}(\mc{O}'')$. So the results in Section 8 of \cite{BV1985} applies and the special unipotent representations attached to $\mc{O}$ are given by
\begin{align} \label{eq:Opi}
X_{\mc{O},\pi} = Ind_{G'' \times GL(\mu_1) \times \dots \times GL(\mu_x) \times GL(\nu_1) \times \dots \times GL(\nu_y)}^{G}(X_{\mc{O}'',\pi} \boxtimes \mathrm{triv} \boxtimes \dots \boxtimes \mathrm{triv}).
\end{align}
\begin{theorem} \label{prop:psigeneral}
Let $\mc{O}$ be a special nilpotent orbit of classical type. Write $\mc{O}$ as
$$\mc{O} = \mc{O}'' \cup (\mu_1, \mu_1, \dots, \mu_x, \mu_x) \cup (\nu_1, \nu_1, \dots, \nu_y, \nu_y),$$
with $\mc{O}''$ to be a nilpotent orbit given in Section \ref{subsec:nomunu}. Then we have
$$\Psi(\mc{O}, \pi) = \Psi(\mc{O}'',\pi) \cup \bigcup_{u=1}^x K_{\mu_u} \cup \bigcup_{v=1}^y K_{\nu_v},$$
where $\Psi(\mc{O}'',\pi)$ is determined in Proposition \ref{prop:psinomunu}, and $K_j$ is given by
\begin{align}\label{eq:Kj}
K_j = \begin{cases}
(j-1,j-1,j-3,j-3,\dots,2,2,0)\ \ \ \text{when}\ $j$\ \text{is odd},\\
(j-1,j-1, j-3,j-3, \dots, 1,1)\ \ \ \text{when}\ $j$\ \text{is even}.
\end{cases}
\end{align}
\end{theorem}
\begin{proof}
We study the right hand side of Equation \eqref{eq:Opi}: The maximal term in $X_{\mc{O}'',\pi}$ is given by $\Psi(\mc{O}'',\pi)$. Also, by the Weyl character formula, the trivial representation of $GL(j)$ is written as
$$\sum_{w \in W(A_{j-1})} \text{sgn}(w) Ind_T^{GL(j)}((\frac{j-1}{2},\frac{j-3}{2},\dots,-\frac{j-1}{2})- w(\frac{j-1}{2},\frac{j-3}{2},\dots,-\frac{j-1}{2})),$$
where the largest term appearing inside the bracket is obtained when $w = w_0$, the longest element in $W(A_{j-1})$, and is equal to (a $W$-conjugate of) $K_{j}$. Using induction in stages upon Equation \eqref{eq:Opi}, the result follows.
\end{proof}


\section{Proof of the Achar-Sommers Conjecture}  \label{sec:AS}
\subsection{Proof of Theorem C}
In order to prove Theorem C, one needs to express $R(\widetilde{\mc{O}}^{C})$ in the form of Equation \eqref{eq:rocchar}. The proposition below gives precisely the formula of $R(\widetilde{\mc{O}}^{C})$:
\begin{proposition}  \label{prop:ROC}
Let $\mc{O}$ be a classical special nilpotent orbit in the form of Proposition \ref{prop:abar}. Given any conjugacy class $C_I := \prod_{i \in I}\theta_{i}$ in $\overline{A}(\mc{O})$,
$$R(\widetilde{\mc{O}}^{C_I}) = \bigoplus_{\pi = \chi_q \boxtimes \dots \boxtimes \chi_1,\  \chi_i = \mathrm{triv}\ \forall i\in I} X_{\mc{O},\pi}.$$
Consequently, the maximal term appearing in the expression of $R(\widetilde{\mc{O}}^{C_I})$ is given by
$$\max\{\Psi(\mc{O},\pi)\ |\ \pi = \chi_q \boxtimes \dots \boxtimes \chi_1,\ \chi_i = \mathrm{triv}\ \forall i\in I\}.$$
\end{proposition}
\begin{proof}
We first study the case when $A(\mc{O}) = \overline{A}(\mc{O})$. By the description of $K_C \leq \overline{A}(\mc{O}) = A(\mc{O})$ in Section \ref{subsec:AS}, one has the following:
$$R(\widetilde{\mc{O}}^{C_I}) = \bigoplus_{\pi(k_C) = 1\ \text{for all}\ k_C \in K_C} R(\mc{O},\pi) = \bigoplus_{\pi = \chi_q \boxtimes \dots \boxtimes \chi_1,\  \chi_i = \mathrm{triv}\ \forall i\in I} R(\mc{O},\pi).$$
By Theorem \ref{thm:unipclassical}, $R(\mc{O},\pi) = X_{\mc{O},\pi}$ for all $\pi$'s, so the Proposition holds in this case.

Suppose now $A(\mc{O}) \neq \overline{A}(\mc{O})$, i.e. there exists column pairs of the form $(\mu,\mu)$ in Proposition \ref{prop:abar}. Let $\mc{O} = \mc{O}' \cup (\mu_1,\mu_1,\dots,\mu_x,\mu_x)$ such that $A(\mc{O}) = \overline{A}(\mc{O}) \times (\bb{Z}/2\bb{Z})^x$ as in Section \ref{subsec:generalize}. By the definition of $H_C = r^{-1}(K_C) = K_C \times (\bb{Z}/2\bb{Z})^x$,
$$R(\widetilde{\mc{O}}^{C_I}) = \bigoplus_{\pi(k_C) = 1\ \text{for all}\ k_C \in K_C} R(\mc{O},\pi \boxtimes \mathrm{triv}^x) = \bigoplus_{\pi = \chi_q \boxtimes \dots \boxtimes \chi_1,\  \chi_i = \mathrm{triv}\ \forall i\in I} R(\mc{O},\pi \boxtimes \mathrm{triv}^x).$$
By the results in Section \ref{subsec:generalize}, $R(\mc{O},\pi \boxtimes \mathrm{triv}^x) \cong X_{\mc{O},\pi}$, therefore the result follows.
\end{proof}

The following Lemma is essential in the proof of Theorem C:
\begin{lemma}  \label{lem:canonical}
The collection of $(\mc{O},C)$ that appears in Sommers' canonical preimage for all classical $\mf{g}$ is given as follows:

\noindent {\bf Type $B_n$:} Let $\mc{O}_B = (a_{2q+1}'' \geq a_{2q}'' \geq \dots \geq a_0'')\cup (\mu_1, \mu_1, \dots, \mu_x, \mu_x) \cup (\nu_1, \nu_1, \dots, \nu_y, \nu_y)$ be a special nilpotent orbit of Type $B_n$. Suppose $C_I = \prod_{i \in I}\theta_{i}$, then $(\mc{O}_B,C_I)$ is the canonical preimage of an orbit $\mc{O}_{B,I}^{\vee}$ iff the partition
$$[a_{2q+1}''-1] \cup \bigcup_{i\in I} [a_{2i}'', a_{2i-1}''] \bigcup_{j \notin I} [a_{2j}''+1,a_{2j-1}''-1] \cup [\mu_1, \mu_1, \dots, \mu_x, \mu_x] \cup [\nu_1, \nu_1, \dots, \nu_y, \nu_y]$$
defines a nilpotent orbit in ${}^L{\mf{g}} = \mf{sp}(2n,\bb{C})$ {\bf in terms of rows} (since every $a_j''$ is odd, this forces $a_{2i}''=a_{2i-1}''$ by Proposition \ref{prop:classrows}). Moreover, $\mc{O}_{B,I}^{\vee}$ is equal to the orbit with the above partition, and all orbits $\mc{O}^{\vee}$ in ${}^L{\mf{g}}$ can be expressed in this form.\\

\noindent {\bf Type $C_n$:} Let $\mc{O}_C = (a_{2q}'' \geq a_{2q-1}'' \geq \dots \geq a_0'')\cup (\mu_1, \mu_1, \dots, \mu_x, \mu_x) \cup (\nu_1, \nu_1, \dots, \nu_y, \nu_y)$ be a special nilpotent orbit of Type $C_n$. Suppose $C_I = \prod_{i \in I}\theta_{i}$, then $(\mc{O}_C,C_I)$ is the canonical preimage of an orbit $\mc{O}_{C,I}^{\vee}$ iff the partition
$$\bigcup_{i\in I} [a_{2i}'', a_{2i-1}''] \bigcup_{j \notin I} [a_{2j}''+1,a_{2j-1}''-1] \cup [a_0''+1]\cup [\mu_1, \mu_1, \dots, \mu_x, \mu_x] \cup [\nu_1, \nu_1, \dots, \nu_y, \nu_y]$$
defines a nilpotent orbit in ${}^L{\mf{g}} = \mf{so}(2n+1,\bb{C})$ {\bf in terms of rows} (since every $a_j''$ are even, this forces $a_{2i}''=a_{2i-1}''$ by Proposition \ref{prop:classrows}). Moreover, $\mc{O}_{C,I}^{\vee}$ is equal to the orbit with the above partition, and all orbits $\mc{O}^{\vee}$ in ${}^L{\mf{g}}$ can be expressed in this form.\\

\noindent {\bf Type $D_n$:} Let $\mc{O}_D = (a_{2q+1}'' \geq a_{2q}'' \geq \dots \geq a_0'')\cup (\mu_1, \mu_1, \dots, \mu_x, \mu_x) \cup (\nu_1, \nu_1, \dots, \nu_y, \nu_y)$ be a special, non-very even nilpotent orbit of Type $D_n$. Suppose $C_I = \prod_{i \in I}\theta_{i}$, then $(\mc{O}_D,C_I)$ is the canonical preimage of an orbit $\mc{O}_{D,I}^{\vee}$ iff the partition
$$[a_{2q+1}''-1] \cup \bigcup_{i\in I} [a_{2i}'', a_{2i-1}''] \bigcup_{j \notin I} [a_{2j}''+1,a_{2j-1}''-1] \cup [a_0''+1]\cup [\mu_1, \mu_1, \dots, \mu_x, \mu_x] \cup [\nu_1, \nu_1, \dots, \nu_y, \nu_y]$$
defines a nilpotent orbit in ${}^L{\mf{g}} = \mf{so}(2n,\bb{C})$ {\bf in terms of rows} (since every $a_j''$ is even, this forces $a_{2i}''=a_{2i-1}''$ by Proposition \ref{prop:classrows}). Moreover, $\mc{O}_{D,I}^{\vee}$ is equal to the orbit with the above partition, and all non-very even orbits $\mc{O}^{\vee}$ in ${}^L{\mf{g}}$ can be expressed in this form.

If $\mc{O}_{I,II} = (2\alpha_k, 2\alpha_k, 2\alpha_{k-1}, 2\alpha_{k-1},\dots, 2\alpha_1, 2\alpha_1)_{I,II}$ are the very even orbits, then $(\mc{O}_I,C_{\phi})$, $(\mc{O}_{II},C_{\phi})$ are both canonical preimages of the very even orbits with row sizes $[2\alpha_k$, $2\alpha_k$, $2\alpha_{k-1}$, $2\alpha_{k-1}$,$\dots$, $2\alpha_1$, $2\alpha_1]$.
\end{lemma}


The proof of Lemma \ref{lem:canonical} will be postponed to the next subsection. Assuming the Lemma, we can prove the Achar-Sommers Conjecture:\\

\noindent {\it Proof of Theorem C}.
We present the proof for Type $B_n$. The proofs for orbits of Type $C_n$ and non-very even orbits of Type $D_n$ are similar. Suppose $\mc{O}^{\vee} = \mc{O}_{B,I}^{\vee}$ has Sommers' canonical preimage $(\mc{O}_B,C_I)$, then by Lemma \ref{lem:canonical}, $h^{\vee}$ is equal to
\begin{align*}
&(a_{2q+1}''-2,a_{2q+1}''-4,\dots,1)  \cup \bigcup_{i\in I} (a_{2i}''-1, a_{2i}''-1,\dots,2,2,0) \cup \\
&\bigcup_{j \notin I} (a_{2j}'',a_{2j}''-2,\dots, a_{2j-1}'',a_{2j-1}''-2,a_{2j-1}''-2,\dots,3,3,1,1) \cup \bigcup K_{\mu} \cup \bigcup K_{\nu},
\end{align*}
where $K_{\mu}$, $K_{\nu}$ are as defined in Equation \eqref{eq:Kj}. On the other hand, by Proposition \ref{prop:ROC}, the maximal term appearing in Equation \eqref{eq:rocchar} is given by the maximum of $\Psi(\mc{O}_B,\pi)$ among all $\pi =$ $\chi_q$ $\boxtimes$ $\dots$ $\boxtimes \chi_1$'s satisfying $\chi_i = \mathrm{triv}$ for all $i \in I$. From the calculations in Theorem \ref{prop:psigeneral}, this is equal to
\begin{align*}
&(a_{2q+1}''-2,a_{2q+1}''-4,\dots,1)  \cup \bigcup_{i\in I} B_i' \cup \\
&\bigcup_{j \notin I} (a_{2j}'',a_{2j}''-2,\dots, a_{2j-1}'',a_{2j-1}''-2,a_{2j-1}''-2,\dots,3,3,1,1) \cup \bigcup K_{\mu} \cup \bigcup K_{\nu},
\end{align*}
where $B_i'$ is as defined in Proposition \ref{prop:psinomunu}. However, we have seen from Lemma \ref{lem:canonical} that $a_{2i}'' = a_{2i-1}''$, so $\frac{1}{4}(a_{2i}'' - a_{2i-1}'') \in \mathbb{N}$ and $B_i'$ $=$ $(a_{2i}''-1, a_{2i}''-1,\dots,2,2,0)$. Hence this value is equal to $h^{\vee}$,  and the Theorem is proved for Type $B_n$.

So we are left to show Theorem \ref{conj:AS} holds for the very even orbits $\mc{O}_{I,II}^{\vee}$ of Type $D_n$. In this case, the canonical preimage is $((\mc{O}^{\vee})^{\bf t}_{I,II},C_{\phi})$ if $n$ is even and $((\mc{O}^{\vee})^{\bf t}_{II,I},C_{\phi})$ if $n$ is odd. One can directly compare the description of $\Psi((\mc{O}^{\vee})^{\bf t}_{I,II},\mathrm{triv})$ in Proposition \ref{prop:psinomunu} with the Dynkin element $h_{I,II}^{\vee}$ of $\mc{O}_{I,II}^{\vee}$ to obtain the desired result. Therefore, the theorem is proved. \qed

\begin{example}  \label{eg:CD}
Let $\mc{O} = (4,4,2,2,0)$ be a nilpotent orbit of Type $C_6$. Then $A(\mc{O}) = \overline{A}(\mc{O}) = \bb{Z}/2\bb{Z}$, with generators $\theta_2 = b_4$, $\theta_1 = b_2b_4$. According to Lemma \ref{lem:canonical}, all $(\mc{O},C)$ are canonical preimages of Sommers' surjection map $d$:
\begin{align*}
(\mc{O}, \phi) &\stackrel{d}{\longrightarrow} \mc{O}^{\vee} = [5,3,3,1,1];\ \ \ h^{\vee} = (4,2,2,0) \cup (2,0)\\
(\mc{O}, \theta_2) &\stackrel{d}{\longrightarrow} \mc{O}^{\vee} = [4,4,3,1,1];\ \ \ h^{\vee} = (3,3,1,1) \cup (2,0)\\
(\mc{O}, \theta_1) &\stackrel{d}{\longrightarrow} \mc{O}^{\vee} = [5,3,2,2,1];\ \ \ h^{\vee} = (4,2,2,0) \cup (1,1)\\
(\mc{O}, \theta_2\theta_1) &\stackrel{d}{\longrightarrow} \mc{O}^{\vee} = [4,4,2,2,1];\ \ \ h^{\vee} = (3,3,1,1) \cup (1,1).
\end{align*}
Note that the partitions on the right hand side always define an orbit of Type $B_6$ in terms of rows. One can also verify the above results using Theorem 12 of \cite{S2} -- for example, take $(\mc{O}, \theta_2\theta_1)$ $= (\mc{O}, b_4\cdot b_2b_4)$ $= ([4,4,2,2],b_2)$. Since $b_2$ corresponds to Row $2$ of $\mc{O} = [4,4,2,2]$ of size $4$, $(\mc{O}, \theta_2\theta_1)$ is written as $([4],[4,2,2])$ in the notations of {\it loc. cit.}. In other words, the first partition in the bracket parametrizes $\overline{A}(\mc{O})$. Therefore, the formula in Theorem 12 of {\it loc. cit.} gives $d(\mc{O},\theta_2\theta_1) = (5,4,2,2) = [4,4,2,2,1]$, which is the same as above.

We now look at the orbit covers for each canonical preimage above. By Proposition \ref{prop:ROC},
\begin{align*}
R(\widetilde{\mc{O}}^{C_{\phi}}) &= X_{\mc{O},\mathrm{triv} \boxtimes \mathrm{triv}} \oplus X_{\mc{O},\mathrm{sgn} \boxtimes \mathrm{triv}} \oplus X_{\mc{O},\mathrm{triv} \boxtimes \mathrm{sgn}} \oplus X_{\mc{O},\mathrm{sgn} \boxtimes \mathrm{sgn}}\\
R(\widetilde{\mc{O}}^{C_{2}}) &= X_{\mc{O},\mathrm{triv} \boxtimes \mathrm{triv}} \oplus X_{\mc{O},\mathrm{triv} \boxtimes \mathrm{sgn}} \\
R(\widetilde{\mc{O}}^{C_{1}}) &= X_{\mc{O},\mathrm{triv}} \oplus X_{\mc{O},\mathrm{sgn} \boxtimes \mathrm{triv}}\\
R(\widetilde{\mc{O}}^{C_{2,1}}) &= X_{\mc{O},\mathrm{triv} \boxtimes \mathrm{triv}}
\end{align*}
By Theorem \ref{prop:psigeneral}, $\Psi(\mc{O},\mathrm{triv} \boxtimes \mathrm{triv}) = (3,3,1,1) \cup (1,1)$, $\Psi(\mc{O},\mathrm{sgn} \boxtimes \mathrm{triv}) = (4,2,2,0) \cup (1,1)$, $\Psi(\mc{O},\mathrm{triv} \boxtimes \mathrm{sgn}) = (3,3,1,1) \cup (2,0)$ and $\Psi(\mc{O},\mathrm{sgn} \boxtimes \mathrm{sgn}) = (4,2,2,0) \cup (2,0)$. One can therefore see that the above orbit covers have maximal term equal to the Dynkin element of $\mc{O}^{\vee} =$ $[5,3,3,1,1]$, $[4,4,3,1,1]$, $[5,3,2,2,1]$ and $[4,4,2,2,1]$ respectively. This verifies Theorem C for these $\mc{O}^{\vee}$'s.
\end{example}


\subsection{Proof of Lemma \ref{lem:canonical}} \label{subsec:lem}
We will prove the Lemma for Type $C_n$, that is, we start with an orbit $\mc{O}^{\vee} = \mc{O}^{\vee} = [r_k, r_{k-1}, \dots, r_1]$ of Type $B_n$. Consider the {\bf $B$-expansion} of $\mc{O}^{\vee}$ in terms of Section 6.3 of \cite{CM}, which gives the smallest special orbit $\mc{O}_{sp}^{\vee}$ above $\mc{O}^{\vee}$. Writing $\mc{O}_{sp}^{\vee}$ as in Proposition \ref{prop:abarrows}, we have
$$\mc{O}_{sp}^{\vee} = [\rho_{2q}'' > \rho_{2q-1}'' \geq \rho_{2q-2}'' > \dots \geq \rho_2'' > \rho_1'' \geq \rho_0''] \cup [\alpha_1, \alpha_1, \dots, \alpha_x, \alpha_x] \cup [\beta_1, \beta_1, \dots, \beta_y, \beta_y]$$
with all $\rho_l''$ being odd integers.

We now specify the discrepancies between $\mc{O}^{\vee}$ and $\mc{O}_{sp}^{\vee}$ -- by Lemma 6.3.9 of \cite{CM}, the discrepancies must occur at some $[\rho_{2i}'' > \rho_{2i-1}'']$ in each of $\mc{O}_{sp}^{\vee}$, that is, upon replacing some of  $[\rho_{2i}'' > \rho_{2i-1}'']$ by $[\rho_{2i}''-1 \geq \rho_{2i-1}''+1]$, we will get back $\mc{O}^{\vee}$. Let $I \subset \{q, \dots, 1\}$ be the subset of all such $i$'s, then
$$\mc{O}^{\vee} = \bigcup_{i \in I}[\rho_{2i}''-1, \rho_{2i-1}''+1] \bigcup_{j \notin I} [\rho_{2j}'', \rho_{2j-1}''] \cup [\rho_0''] \cup [\alpha_1, \alpha_1, \dots, \alpha_x, \alpha_x] \cup [\beta_1, \beta_1, \dots, \beta_y, \beta_y].$$

Now use Proposition 13 of \cite{S2} to compute the canonical preimage of $\mc{O}^{\vee}$. In fact, the canonical preimage must be of the form $(\mc{O}_{sp},C)$, where
$$\mc{O}_{sp} = (\rho_{2q}''-1, \rho_{2q-1}''+1, \rho_{2q-2}''-1, \dots, \rho_2''-1, \rho_1''+1, \rho_0''-1) \cup (\alpha_1, \alpha_1, \dots, \alpha_x, \alpha_x) \cup (\beta_1, \beta_1, \dots, \beta_y, \beta_y)$$
is the Lusztig-Spaltenstein dual of $\mc{O}_{sp}^{\vee}$, and $C$ is a conjugacy class of $\overline{A}(\mc{O}_{sp})$ parametrized by $\theta_q = b_{\rho_{2q-1}''+1}$ and $\theta_i = b_{\rho_{2i+1}''+1}b_{\rho_{2i-1}''+1}$ for $q-1 \geq i \geq 1$. Note that our expression of $\mc{O}_{sp}$ above is compatible with the column expression of special orbits given in Proposition \ref{prop:abar}, with $\alpha$'s and $\beta$'s acting as $\mu$'s and $\nu$'s.

To find what $C$ is, consider the transpose of $\mc{O}^{\vee}$ given by
$$(\mc{O}^{\vee})^{\bf t} = \bigcup_{i \in I}(\rho_{2i}''-1, \rho_{2i-1}''+1) \bigcup_{j \notin I} (\rho_{2j}'', \rho_{2j-1}'') \cup (\rho_0'') \cup (\alpha_1, \alpha_1, \dots, \alpha_x, \alpha_x) \cup (\beta_1, \beta_1, \dots, \beta_y, \beta_y),$$
and check whether $(\mc{O}^{\vee})^{\bf t}$ defines an orbit of Type $B_n$. If it does, then $\mc{O}^{\vee}$ is special (see Proposition \ref{prop:classrows}), and the canonical preimage of $\mc{O}^{\vee}$ is $(\mc{O}_{sp},\phi)$. If not, we remove some distinct even rows from $(\mc{O}^{\vee})^{\bf t}$ so that the remaining parts form an orbit of Type $B_n$, and the removed rows form a partition that determines a conjugacy class $C \subset \overline{A}(\mc{O}_{sp})$ as in Example \ref{eg:CD} above.

Rather than writing down the sizes of the removed rows explicitly, we record the row numbers of $(\mc{O}^{\vee})^{\bf t}$ that are removed. By the discussion at Section \ref{subsec:generators}, if we are to determine $C$ as a conjugacy class of the Lusztig quotient $\overline{A}(\mc{O}_{sp})$, we can ignore the $\alpha$ or $\beta$ columns in $(\mc{O}^{\vee})^{\bf t}$. Moreover, for each Row $\rho_{2i-1}''+1$ or Row $\rho_{2j-1}''$ removed from $(\mc{O}^{\vee})^{\bf t}$ (with $\alpha$ and $\beta$ columns omitted), it contributes a factor $b_{\rho_{2i-1}''+1}$ or $b_{\rho_{2j-1}''+1}$ to $C$.

Omitting the $\alpha$ and $\beta$ columns of $(\mc{O}^{\vee})^{\bf t}$, we now see which rows of $(\mc{O}^{\vee})^{\bf t}$ need to be removed in order to define an orbit of Type $B_n$. By the classification of nilpotent orbits in terms of columns in Proposition \ref{prop:class}, it does not define an orbit of Type $B_n$ precisely when there exists two even columns $(\rho_{2i}''-1, \rho_{2i-1}''+1)$ for some $i \in I$. Let $i_1$ be the smallest integer in $I$, then Row $\rho_{2i_1-1}''+1$ must be removed from $(\mc{O}^{\vee})^{\bf t}$, therefore $b_{\rho_{2i_1-1}''+1} = \theta_{i_1}\theta_{i_1+1}\dots\theta_q$ contributes to $C$.

Now consider the second smallest integer $i_2$ in $I$. If $i_2 = i_1 +1$, then $(\mc{O}^{\vee})^{\bf t}$ contains columns $(\rho_{2i_1+2}''-1, \rho_{2i_1+1}''+1)$ of larger sizes than $(\rho_{2i_1}''-1, \rho_{2i_1-1}''+1)$. After removing Row $\rho_{2i-1}''+1$ from $(\mc{O}^{\vee})^{\bf t}$, the columns $(\rho_{2i_1+2}''-1, \rho_{2i_1+1}''+1)$ become $(\rho_{2i_1+2}''-2, \rho_{2i_1+1}'')$ which are both odd-sized. So $b_{\rho_{2i_1+1}''+1}$ does not contribute to $C$, and $C = \theta_{i_1}\theta_{i_1+1}\prod \theta_j$ for some $j > i_1+1$. On other other hand, if $i_2 > i_1+1$, then $b_{\rho_{2i_1+1}''+1} = \theta_{i_1+1}\dots\theta_q$ contributes to $C$ and $C = \theta_{i_1} \prod \theta_j$ for some $j > i_1+1$. In other words, if $i_1+1 \in I$, then $\theta_{i_1+1}$ shows up in $C$ and vice versa.

One can continue these arguments to conclude that $C = \prod_{i \in I} \theta_i = C_I$, that is, $(\mc{O}_{sp},C_I)$ is the canonical preimage of $\mc{O}^{\vee}$. Then Lemma \ref{lem:canonical} follows directly by replacing $\rho_{2l}''$ with $a_{2l}''+1$ and $\rho_{2l-1}'' $ with $a_{2l-1}'' - 1$ in the above expressions of $\mc{O}^{\vee}$ and $\mc{O}_{sp}$. \qed





\end{document}